\documentclass[11pt]{amsart}
\title{ Weierstrass semigroups on Castelnuovo curves }
\author[N. Pflueger]{Nathan Pflueger}\address{Department of Mathematics and Statistics, Amherst College}\email{npflueger@amherst.edu}
\date{\today}

\usepackage{color}
\usepackage{xifthen}
\usepackage{amssymb}
\usepackage{amsmath}
\usepackage{enumerate}
\usepackage{tikz}

\usepackage{multirow}

\newtheorem{thm}{Theorem}[section]
\newtheorem{lemma}[thm]{Lemma}
\newtheorem{prop}[thm]{Proposition}

\newtheorem{cor}[thm]{Corollary}

\newtheorem{conj}[thm]{Conjecture}

\theoremstyle{definition}
\newtheorem{defn}[thm]{Definition}
\newtheorem{innereg}[thm]{Example}

\newtheorem{rem}[thm]{Remark}

\newtheorem{qu}[thm]{Question}

\newcommand{\AAA}{\textbf{A}}
\newcommand{\CC}{\textbf{C}}
\newcommand{\FF}{\textbf{F}}
\newcommand{\NN}{\textbf{N}}
\newcommand{\PP}{\textbf{P}}
\newcommand{\ZZ}{\textbf{Z}}

\newcommand{\iou}[1][]{
    \ifthenelse{\equal{#1}{}}{{\color{blue}\{IOU\}}}
    {{\color{blue}\{IOU: #1\}}}
}

\newcommand{\cG}{\mathcal{G}}
\newcommand{\cH}{\mathcal{H}}
\newcommand{\cM}{\mathcal{M}}
\newcommand{\cO}{\mathcal{O}}

\newcommand{\cw}{\mathcal{CW}}

\newcommand{\fd}{\FF_\delta}
\newcommand{\eps}{\varepsilon}
\newcommand{\psirmed}{\Psi_{r,m,\eps}^\delta}
\newcommand{\psiver}{\Psi_{5,d}^{\textrm{ver}}}

\DeclareMathOperator{\Spec}{Spec}
\DeclareMathOperator{\Aut}{Aut}
\DeclareMathOperator{\ewt}{ewt}
\DeclareMathOperator{\codim}{codim}
\DeclareMathOperator{\scroll}{Sc}

\begin{document}

\begin{abstract}
We define a class of numerical semigroups $S$, which we call Castelnuovo semigroups, and study the subvariety $\cM^S_{g,1}$ of $\cM_{g,1}$ consisting of marked smooth curves with Weierstrass semigroup $S$. We determine the number of irreducible components of these loci and determine their dimensions. Curves with these Weierstrass semigroups are always Castelnuovo curves, which provides the basic tool for our argument. This analysis provides examples of numerical semigroups for which $\cM^S_{g,1}$ is reducible and non-equidimensional.
\end{abstract}

\maketitle


\section{Introduction} \label{s_intro}

If $C$ is a smooth algebraic curve of genus $g$, and $p$ is any point on $C$, there is a numerical semigroup $S(C,p)$, called the \emph{Weierstrass semigroup}, given by the set of pole orders at $p$ of rational functions on $C$ that are regular away from $p$. Weierstrass's L\"{u}ckensatz (a consequence of the Riemann-Roch formula) states that this semigroup has complement of size $g$ in $\NN$; therefore the number of gaps in a numerical semigroup is called its \emph{genus}.

If $(C,p)$ is a general marked smooth curve, $S(C,p)$ is the ``ordinary semigroup'' $\{0,g+1,g+2,\cdots\},$ and the $i$th gap of the Weierstrass semigroup is upper semicontinuous in $\cM_{g,1}$. Thus the numerical semigroups of genus $g$ give a stratification of $\cM_{g,1}$ by locally closed subvarieties $\cM^S_{g,1}$.  Many aspects of this stratification remain mysterious, especially in high codimension.

The purpose of this paper is to describe the irreducible components of $\cM^S_{g,1}$ in the case of a special class of semigroups: those whose generators form a contiguous interval in $\NN$. For positive integers $r,d$ such that $r \geq 2$ and $d \geq 2r-1$, we denote
$$
S_{r,d} = \langle d-r+1, d-r+2, \cdots, d \rangle.
$$

For reasons to be explained shortly, we call these \emph{Castelnuovo semigroups}. In addition to providing a large family of high-codimension strata in $\cM_{g,1}$, these semigroups also provide infinitely many examples in which $\cM^S_{g,1}$ is reducible, and indeed non-equidimensional. We prove the following.

\begin{thm} \label{t_intro}
Suppose that $r \geq 3$ and $d \geq 2r+1$. Denote the genus of $S$ by $g$.
\begin{enumerate}
\item The locus $\cM^{S_{r,d}}_{g,1}$ is nonempty, and each of its irreducible components $X$ satisfies
$$
g+3 \leq \dim X \leq \frac43 g + 2.
$$
\item If $r \geq 4$ and $(r-1) | (d-1)$, or if $r=5$ and $d$ is even, then $\cM^{S_{r,d}}_{g,1}$ has two irreducible components of different dimensions. Otherwise, $\cM^{S_{r,d}}_{g,1}$ is irreducible.
\end{enumerate}
\end{thm}

In fact, we compete the exact dimension of the components of $\cM^{S_{r,d}}_{g,1}$; see Theorem \ref{t_main}. Theorem \ref{t_intro} as stated will be proved in Section \ref{s_msg} as a straightforward corollary. The proof is based on Castelnuovo theory: we show that if $(C,p) \in \cM^{S_{r,d}}_{g,1}$, then $C$  can be embedded as a Castelnuovo curve of degree $d$ in $\PP^r$ (uniquely up to $\Aut \PP^r$), such that the osculating plane at $p$ meets $C$ at no other points. The components of the Hilbert scheme of Castelnuovo curves are known (\cite{ciliberto}; discussed without proof in \cite{montreal}), and the technical aspect of our proof is to establish that imposing the vanishing condition at $p$ drops the dimension by the expected amount.

The assumptions $r \geq 3,\ d \geq 2r+1$ are merely to exclude a few exceptional cases, which are already well-understood:

\begin{enumerate}
\item If $d=2r-1$, then $S_{r,d} = \{0,r,r+1,r+2,\cdots\}$, i.e. the ordinary semigroup of genus $g = r-1$. So $\cM^{S_{r,d}}_{g,1}$ is a dense open subset of $\cM_{g,1}$.
\item If $d=2r$, then $S_{r,d}$ has gaps $\{1,2,\cdots, r, 2r+1\}$, so $g=r+1$. This semigroup is studied in \cite{bullock}, where it is proved that $\cM^{S_{r,d}}_{g,1}$ is irreducible of dimension $2g-1$; it is the locus denoted $\cG^{\textrm{odd}}_g$ in \cite{bullock}, corresponding to points $p$ on non-hyperelliptic curves such that $(g-1)p$ is an odd theta characteristic.
\item If $r=2$, then $S = S_{2,d}$ has only two generators. Such semigroups are easy to analyze; see for example \cite[Section 2.4]{pfl}. The genus is $g = \binom{d-1}{2}$, and $\cM^S_{g,1}$ is irreducible of dimension $g + 2d - 6$.
\end{enumerate}

The original impetus to study Castelnuovo semigroups arose by studying the \emph{effective weight} of numerical semigroups, which is an upper bound on the codimension of $\cM^S_{g,1}$ in $\cM_{g,1}$ introduced in \cite{pfl}. It is defined by $$\ewt(S) = \sum_{b \in \NN \backslash S} \# \left( \mbox{generators $a$ of $S$ with $a < b$} \right).$$ Attempting to find semigroups for which $\codim \cM^S_{g,1} < \ewt(S)$ led the author to attempt to classify the numerical semigroups of maximum effective weight. Experimental evidence strongly suggests that this maximum is $\left\lfloor \frac{(g+1)^2}{8} \right\rfloor$, and that for $g \geq 10$ this maximum is attained only by a particular set of Castelnuovo semigroups. We discuss these matters in more detail in Section \ref{s_ew}.

The three cases above show that if $d \leq 2r$ or $r=2$, then $\cM^{S_{r,d}}_{g,1}$ has codimension $\ewt(S)$ in $\cM_{g,1}$. In fact, we prove (Proposition \ref{p_effproper}) that the same is true if $d = 2r+1$ and a few other cases, but that $\codim \cM^{S_{r,d}}_{g,1} < \ewt(S_{r,d})$ for all other $r,d$. This is the largest class of such semigroups that the author is aware of, and will hopefully suggest how the quantity $\ewt(S)$ can be improved to better predict $\dim \cM^{S}_{g,1}$ for a broader range of semigroups.

\section*{Conventions}

We work over the field $\CC$ of complex numbers. A \emph{point} of a scheme will always refer to a closed point, and when we say that a \emph{general point} of a scheme satisfies a property, we mean that there exists a dense open subset in which all points satisfy the property. A \emph{curve} is always reduced, connected, and complete.

If $D$ is a divisor on a variety $X$, then $|D|$ will denote the complete linear series of $D$, and $\phi_{|D|}$ denotes the associated map to projective space.

\section*{Acknowledgments}

This paper is an extension of a portion of my thesis at Harvard University. I am grateful to my adviser Joe Harris for many suggestions and helpful conversations both during and after graduate school.

\section{Fully inflected Castelnuovo curves} \label{s_ficc}

Throughout this section we fix two integers $r,d$, and assume that $r \geq 3$ and $d \geq 2r+1$.
The genus of the semigroup $S_{r,d}$ is a quantity that arises in classical algebraic geometry.

\begin{defn} \label{d_pird}
For integers $r,d$ with $r \geq 2$ and $d \geq 2r-1$, let $m = \lfloor \frac{d-1}{r-1} \rfloor$ and $\eps = (d-1) - m(r-1)$. Define
$$
\pi(r,d) = \binom{m}{2} (r-1) + m \eps.
$$ 
\end{defn}

\begin{rem} \label{r_evals}
If $r-1$ divides $d-1$, then we can set $(m,\eps)$ to either $(\frac{d-1}{r-1}, 0)$ or $(\frac{d-1}{r-1}-1,r-1)$ in the expression on the right hand side, and obtain the same result. Going forward, it will often be convenient, instead of fixing integers $r,d$, to fix integers $r,m,\eps$ such that $0 \leq \eps \leq r-1$, and to then define $d = m(r-1) + \eps+1$, therefore splitting cases where $d-1$ is divisible by $r-1$ into two separate cases. The same expression $\binom{m}{2}(r-1) + m\eps$ for $\pi(r,d)$ can be used in both cases. Furthermore, note that if we fix $m,r,\eps$ in this way, our standing assumption that $r \geq 3$ $d \geq 2r+1$ is equivalent to saying that $r \geq 3$ and either $m \geq 3$ or $m=\eps = 2$ (this is easily seen by observing that $d-2r-1 = (m-2)(r-1) + \eps -2$).
\end{rem}

\begin{lemma} \label{l_genus}
The genus of $S_{r,d}$ is $\pi(r,d)$.
\end{lemma}
\begin{proof}
The gaps of $S_{r,d}$ can be decomposed into a sequence of contiguous segments of the form $(kd , (k+1) (d-r+1)) \cap \ZZ$. The sizes of these segments form an arithmetic progression, beginning at $d-r$, with common difference $-(r-1)$, and terminating once the next term would be negative. An elementary calculation shows that the sum of this progression is $\pi(r,d)$.
\end{proof}

A classical theorem of Castelnuovo (see e.g. \cite[p. 527]{gh}, \cite[Chapter 3]{montreal}, or \cite[\S III.2]{ACGH}) states that a smooth non-degenerate curve of degree $d$ in $\PP^r$ has genus at most $\pi(r,d)$. A curve achieving this maximum for $r \geq 3$ and $d \geq 2r+1$ is called a \emph{Castelnuovo curve.}

\begin{rem}
Carvalho and Torres \cite{ct} prove an upper bound on the genus of a numerical semigroup that is analogous to Castelnuovo's bound on the genus of a curve in $\PP^r$. Castelnuovo semigroups give an equality case for this bound.
\end{rem}

\begin{lemma} \label{l_embeds}
If $(C,p) \in \cM^{S_{r,d}}_{g,1}$, then the complete linear system $|d \cdot p|$ embeds $C$ in $\PP^r$ as a Castelnuovo curve of degree $d$, and the contact orders of hyperplanes at $p$ are exactly $0,1,2,\cdots,r-1$, and $d$.
\end{lemma}
\begin{proof}
Regard the elements of $H^0(C,\cO_C(d \cdot p))$ as regular functions on $C \backslash \{p \}$. The ring generated by these functions includes functions of every possible pole order at $d$, hence this ring includes all regular functions on $C \backslash \{p\}$. So $\phi_{|d \cdot p|}$ embeds $C \backslash \{p\}$ in $\AAA^r$. Only $p$ is sent to infinity, so $\phi_{|d \cdot p|}$ is injective on points. The multiplicity of $p$ in an element of $|d \cdot p|$ is equal to the contact order of a hyperplane in $\PP^r$, and in turn it is $d$ minus the pole order of a rational function that is regular away from $p$. So these contact orders are $\{0,1,\cdots,r-1,d\}$. In particular, since $1$ is a contact order at $p$, the map $\phi_{|d \cdot p |}$ induces an injection on the tangent space at $p$; it follows that it is an embedding.
\end{proof}

\begin{lemma} \label{l_reverse}
Suppose that $C$ is a Castelnuovo curve of degree $d$ in $\PP^r$, such that the hyperplanes in $\PP^r$ have contact orders exactly $0,1,2\cdots, r-1, d$ at $p$. Then $S(C,p) = S_{r,d}$, and $C$ is embedded by $\phi_{|d \cdot p|}$.
\end{lemma}
\begin{proof}
Let $\zeta$ be the hyperplane pulling back to $d \cdot p$ on $C$. Then pulling back $H^0(\cO_{\PP^r}, \cO_{\PP^r}(\zeta))$ to $C$, and regarding the resulting sections of $H^0(C, \cO_C(d \cdot p))$ as rational functions on $C$ regular away from $p$, we see that $S(C,p)$ contains the elements $d-r+1,d-r+2,\cdots,d$. Hence $S_{r,d} \subseteq S(C,p)$. Since $S(C,p)$ has the same number of gaps, $g$, as $S_{r,d}$, these two semigroups are equal. The fact that $C$ is embedded by $\phi_{|d \cdot p|}$ follows since $d\cdot p$ is a hyperplane section and $\dim |d \cdot p| = r$.
\end{proof}

Therefore we can study the irreducible components of $\cM^{S_{r,d}}_{g,1}$ in terms of the irreducible components of the following variety.

\begin{defn} \label{d_Phird}
Let $\Phi_{r,d}$ denote the variety of triples $(C,\zeta,p)$, where $[C]$ is a point in the Hilbert scheme of degree $d$ smooth curves of genus $g$ in $\PP^r$, $p$ is a point on $C$, and $\zeta$ is a hyperplane in $\PP^r$, such that $C$ and $\zeta$ meet only at $p$ and the other contact orders of hyperplanes at $p$ are $0,1,2,\cdots,r-1$.
\end{defn}

\begin{cor} \label{c_mapToMg}
The forgetful map $\Phi_{r,d} \rightarrow \cM_{g,1}$ has image equal to $\cM^{S_{r,d}}_{g,1}$, and all fibers isomorphic to $\Aut \PP^r$. The irreducible components of $\Phi_{r,d}$ are in bijection with the irreducible components of $\cM^{S_{r,d}}_{g,1}$, and each has dimension $\dim \Aut \PP^r$ more than the corresponding component in $\cM^{S_{r,d}}_{g,1}$.
\end{cor}

\subsection{The Hilbert scheme of Castelnuovo curves} \label{ss_hilbert}
We now review the basic structure theorems for Castelnuovo curves.

The standard proof of the Castelnuovo bound has a qualitative consequence: a Castelnuovo curve always lies on a surface of degree $r-1$. A theorem of Bertini (see e.g. \cite{scrollitude}) states that such a surface is either a rational normal surface scroll or a Veronese surface in $\PP^5$. 

By a ``rational normal surface scroll,'' we mean either a smooth scroll or a cone over a rational normal curve of degree $r-1$. Equivalently, a scroll is the variety $\scroll_{a,b}$ \label{pg_sc} cut out by the $2 \times 2$ minors of the matrix
$$
\left( \begin{array}{cccc cccc}
X_0 & X_1  & \cdots & X_{a-1} & Y_{0} & Y_{1} & \cdots & Y_{b-1} \\
X_1 & X_2 & \cdots & X_a & Y_{1} & Y_{2} & \cdots & Y_{b}
\end{array} \right),
$$
where $\{X_0,\cdots,X_a,Y_0,\cdots,Y_b\}$ is a basis for $H^0(\PP^r, \cO_{\PP^r}(1))$. If $a,b$ are positive then the scroll is smooth and isomorphic to the Hirzebruch surface $\FF_{|b-a|}$, while if $a=0$ and $b=r-1$ then the scroll is a cone over a rational normal curve of degree $r-1$, and the blow-up of the scroll at the cone point $[1,0,0,\cdots,0]$ is isomorphic to $\FF_{r-1}$. See Section \ref{s_hirzebruch} for a definition of $\FF_\delta$ and a description of the map to $\PP^r$. The number $\delta = |b-a|$ will be called the \emph{invariant} of the scroll. A scroll with $\delta \in \{0,1\}$ is called \emph{balanced}, while a scroll with $\delta = r-1$ is called a \emph{singular scroll}.

Denote by $\cH_{r,d}$ the (open) locus in the Hilbert scheme of degree $d$ and arithmetic genus $\pi(r,d)$ schemes in $\PP^r$ consisting of smooth non-degenerate curves. A description of the schemes $\cH_{r,d}$ was stated without full proof in \cite{montreal}, and proved in \cite{ciliberto}. We use different notation for these components from that of \cite{ciliberto}, which will be convenient later.

\begin{defn} \label{d_hmre}
Let $r,m,\eps$ be integers such that $r \geq 3$, $0 \leq \eps \leq r-1$, $m \geq 2$. If $m=2$, also assume that $\eps \geq 2$. Denote by $\cH_{r,m,\eps}$ the subvariety of the Hilbert scheme given by curves of class $(m+1) H - (r-2-\eps)F$ on a scroll in $\PP^r$, where $H$ is the hyperplane class and $F$ is the class of a ruling line (see Section \ref{ss_fdPrelims}).

Denote by $\cH^\delta_{r,m,\eps} \subseteq \cH_{r,m,\eps}$ the locally closed subvariety consisting of curves lying on a scroll of invariant $\delta$.
\end{defn}
\begin{defn} \label{d_hver}
Let $d \geq 12$ be an even integer. Let $\cH^{\textrm{ver}}_{5,d}$ denote the locally closed subvariety of $\cH_{5,d}$ consisting of curves of class $\frac{d}{2} H$ on a Veronese surface in $\PP^5$.
\end{defn}

\begin{thm}[{\cite[Theorem 1.4]{ciliberto}}] \label{t_ciliberto}
Let $m,\eps$ be the quotient and remainder when $d-1$ is divided by $r-1$. One irreducible component of $\cH_{r,d}$ is $\cH_{r,m,\eps}$. There is a second irreducible component in the following two cases; in all other cases $\cH_{r,d}$ is irreducible.
\begin{enumerate}
\item If $\eps = 0$ and $r \geq 4$, then $\cH_{r,m-1,r-1}$ is a second irreducible component. Its intersection with the first is equal to $\cH^{r-1}_{r,m,0}$.
\item If $r=5$ and $d$ is even, then $\cH_{5,d}$ is a disjoint union of $\cH_{5,m,\eps}$ and $\cH^{\textrm{ver}}_{5,d}$, and the closure of $\cH^{\textrm{ver}}_{5,d}$ in $\cH_{5,d}$ is a second irreducible component. If $4 | d$ (i.e. $\eps = 3$) then the two components intersect in $\cH^4_{5,m,3}$; otherwise, they are disjoint.
\end{enumerate}
The dimensions of these components are as follows.
\begin{eqnarray*}
\dim \cH_{r,m,\eps} &=& \pi(r,d) + 2m + \eps + d -r -3 + \dim \Aut \PP^r\\
\dim \cH^{\textrm{ver}}_{5,d} &=& \pi(5,d) + 2m + \eps + d-8 + \dim \Aut \PP^r - \frac{\eps+1}{2}
\end{eqnarray*}
\end{thm}

The classification of the components of $\cH_{r,d}$ allows us to split $\Phi_{r,d}$ into pieces according to the type of surface the curve lies on, and its class on that surface.

\begin{defn} \label{d_Phirme}
Denote by $\Phi_{r,m,\eps}$ the parameter space of triples $(C,\zeta,p)$, where $C$ is a curve of class $(m+1)H - (r-2-\eps)F$ on a rational normal surface scroll $X$ in $\PP^r$, $\zeta$ is a hyperplane, and $p$ is a point on $C$ such that $\zeta \cap C = \{p\}$ set theoretically, and such that other hyperplanes in $\PP^r$ have contact of orders $0,1,2,\cdots,r-1$ at $p$.

Denote by $\Phi^{\delta}_{r,m,\eps}$ the locally closed subvariety of $\Phi_{r,m,\eps}$ consisting of triples $(C,\zeta,p)$ such that the scroll on which $C$ lies has invariant $\delta$.

For an even integer $d \geq 12$, let $\Phi^{\textrm{ver}}_{5,d}$ denote the parameter space of triples $(C,\zeta,p)$, where $C$ is a curve of class $\frac{d}{2} H$ on a Veronese surface, $\zeta,p$ are as in the first paragraph, and such that other hyperplanes have contact order $0,1,2,\cdots,r-1$ at $p$.
\end{defn}

We mention one dimension estimate that will be needed later.

\begin{lemma} \label{l_dimbound}
All irreducible components of $\Phi_{r,m,\eps}$ have dimension at least $\pi(r,d)+2m + \eps -2 + \dim \Aut \PP^r$.
\end{lemma}
\begin{proof}
By Theorem \ref{t_ciliberto}, the irreducible component $\cH_{r,m,\eps}$ of $\cH_{r,d}$ is irreducible of dimension $\pi(r,d)+2m+\eps+d-r-3 + \dim \Aut \PP^r$. Therefore the parameter space of triples $(C,\zeta,p)$, where $C$ is such a curve, $\zeta$ is \emph{any} hyperplane in $\PP^r$, and $p$ is point on $C$, has dimension $g+2m+\eps + d - 2 + \dim \Aut \PP^r$. The condition that $\zeta$ meets $C$ to order $d$ at $p$ is given locally by $d$ equations, hence imposing this condition decreasing the dimension by at most $d$. This gives the claimed bound.
\end{proof}

We wish to show that the parameter spaces $\Phi_{r,m,\eps}$ and $\Phi^{\textrm{ver}}_{5,d}$ are irreducible, and to calculate their dimensions exactly. We will do this by first fixing a particular surface and analyzing curves on that surface alone; this is the task of the next two sections.

\section{Analysis on a fixed Veronese surface} \label{s_veronese}

The analysis of $\Phi^{\textrm{ver}}_{5,d}$ is easier than that of $\Phi^\delta_{r,m,\eps}$, so we present it first.

Let $d \geq 12$ be an even integer, and let $k = \frac12 d$ and $g = \binom{k-1}{2}$. We will analyze smooth curves of degree $k$ on $\PP^2$, with a marked point $p$ and contact orders $0,1,2,3,4,d$ with the linear system $|\cO_{\PP^2}(2)|$ of conics. Note that, when $\PP^2$ is embedded in $\PP^5$ as a Veronese surface, the linear system of conics becomes the linear system of hyperplane sections.

\begin{defn} \label{d_Psiver}
Denote by $\psiver$ the sub-variety of $|\cO_{\PP^2}(k)| \times |\cO_{\PP^2}(2)| \times \PP^2$ consisting of triples $(C,\zeta,p)$, where $C$ is a smooth degree $k$ curve, $p \in C$ is a point, and $\zeta$ is a degree $2$ divisor on $\PP^2$ pulling back to $d \cdot p$ on $C$, such that other elements of the linear system $|\cO_{\PP^2}(2)|$ meet $C$ to orders $0,1,2,3$, and $4$ at $p$.
\end{defn}

\begin{lemma} \label{l_verZetaSmooth}
For all $(C,\zeta,p) \in \psiver$, $\zeta$ is a smooth conic.
\end{lemma}
\begin{proof}
Suppose that $\zeta$ is not a smooth conic. Then it must be either a pair of lines or a double line. It follows that there exists a line meeting $C$ to order $k$ at $p$. By choosing another degree $2$ divisor $\zeta'$ consisting of a line meeting $C$ to order $k$ at $p$ and a line not meeting $p$, we see that $k$ is a contact order of the linear system of conics at $p$. But $4 < k < d$, so this is impossible.
\end{proof}

\begin{lemma} \label{l_verContacts}
If $C$ is a smooth degree $k$ curve meeting a smooth conic $\zeta$ to order $d$ at $p$, then there exist other elements of $|\cO_{\PP^2}(2)|$ meeting $C$ to orders $0,1,2,3$, and $4$ at $p$. In other words, $(C,\zeta,p) \in \psiver$.
\end{lemma}
\begin{proof}
The tangent line to $C$ at $p$ must have contact order $2$ with $C$, since $C$ matches $\zeta$ to order greater than $2$ and the tangent line can only meet $\zeta$ to order $2$. Choosing elements of $|\cO_{\PP^2}(2)|$ consisting of a pair of lines, each either disjoint from $p$, transverse to $C$ at $p$, or tangent to $C$ at $p$, we obtain degree $2$ divisors on $\PP^2$ meeting $C$ to order $0,1,2,3,4$ at $p$.
\end{proof}

\begin{cor} \label{c_verPsiDesc}
A triple $(C,\zeta,p) \in |\cO_{\PP^2}(k)| \times |\cO_{\PP^2}(2)| \times \PP^2$ lies in $\psiver$ if and only if $C$ is smooth, $p \in C$, and $C$ meets $\zeta$ to order $d$ at $p$.
\end{cor}

\begin{prop} \label{p_verPsiDim}
Let $m,\eps$ be the quotient and remainder when $d-1$ is divided by $r-1$. The variety $\psiver$ is irreducible of dimension $g+2m + \frac{\eps-1}{2} + 6$.
\end{prop}
\begin{proof}
Fix a smooth conic $\zeta$ and a point $p \in \zeta$. The linear system $|\cO_{\PP^2}(k)|$ pulls back to $\zeta$ as the complete linear system of degree $d$, since $h^1(\PP^2,\cO_{\PP^2}(k-2)) = 0$ \cite[Theorem 5.1]{hartshorne}. So there is a codimension $d$ sub-series $\mathfrak{d}$ of $|\cO_{\PP^2}(k)|$ consisting of divisors $C$ that either meet $\zeta$ to order $d$ at $p$ or contain $\zeta$ entirely.

The only base points of $\mathfrak{d}$ must lie on $\zeta$, since $\mathfrak{d}$ includes all divisors of the form $\zeta + C'$, where $C'$ is a degree $k-2$ curve, and $|\cO_{\PP^2}(k-2)|$ is base point free. Since $\mathfrak{d}$ includes some elements not containing $\zeta$ entirely, and hence meeting $\zeta$ only at $p$ itself, the only base point of $\mathfrak{d}$ is the point $p$.

Bertini's theorem \cite[Corollary 10.9]{hartshorne} implies that a general element of $\mathfrak{d}$ is smooth, except possibly at the point $p$. But the elements $\zeta + C'$ are smooth at $p$ unless $C'$ meets $p$, so a general element of $\mathfrak{d}$ is smooth at $p$ itself. Hence $\mathfrak{d}$ has a dense open subset of smooth curves.

By Corollary \ref{c_verPsiDesc}, a dense open subset of $\mathfrak{d}$ is a fiber of the forgetful map
$$
\psiver \rightarrow \{ (\zeta,p):\ \zeta \in |\cO_{\PP^2}(2)| \mbox{ smooth, } p \in \zeta \}.
$$
Hence each fiber of this map is irreducible of dimension $h^0(\cO_{\PP^2}(k)) - 1 - d = \binom{k+2}{2} - 2k - 1$. Therefore $\psiver$ is irreducible of dimension $\binom{k+2}{2} - 2k + 5$. 
This is equal to $g+2m + \frac{\eps-1}{2} + 6.$
\end{proof}

\section{Analysis on a fixed Hirzebruch surface} \label{s_hirzebruch}

After reviewing some basic information about Hirzebruch surfaces, we perform in this section an analysis on Hirzebruch surfaces similar to the analysis performed in Section \ref{s_veronese} on Veronese surfaces. In contrast to Section \ref{s_veronese}, we will not give an exact dimension calculation in all cases; instead we will give an exact calculation when $\delta \in \{0,1\}$ and an upper bound in the other cases. This upper bound will be enough to show that the locus in $\Phi_{r,m,\eps}$ for which $C$ lies on a balanced scroll is dense.

\subsection{Preliminaries on $\FF_\delta$} \label{ss_fdPrelims}
Rational normal scrolls with invariant $\delta$ are convenient to analyze intrinsically as the Hirzebruch surfaces $\FF_\delta$. A convenient representation of $\FF_\delta$ is as the toric surface defined by the complete fan with $1$-dimensional rays spanned by $(1,0),(0,1),(-1,\delta)$, and $(0,-1)$ (see \cite{fulton} for background on toric varieties; the case of $\FF_\delta$ is discussed on p. 7). The surface is the union of four affine charts: $U_1= \Spec \CC[x,y],\ 
U_2 = \Spec \CC[x,y^{-1}],\ 
U_3 = \Spec \CC[x^{-1}, x^{-\delta} y^{-1}],$ and 
$U_4 = \Spec \CC[x^{-1}, x^\delta y],$
where any two of these charts are glued along the spectrum of the span of their coordinate rings within $\CC[x^{\pm 1}, y^{\pm 1}]$ in the obvious way.

There is a dense open torus $\Spec \CC[x^{\pm 1}, y^{\pm 1}]$ in $\fd$. If $a,b$ are nonnegative integers with $b-a = \delta$, there is a map from this torus to the scroll $\scroll_{a,b}$ (see page \pageref{pg_sc}), given on the level of rings by $X_i / X_0 \mapsto x^i,\ Y_i / X_0 \mapsto x^i y$. It is routine to check that this map extends to a morphism $\fd \rightarrow \scroll_{a,b}$, and that this map is either an isomorphism (if $a > 0$) or the blow-up of the cone point $[1,0,\cdots,0]$ (if $a=0$). See also Corollary \ref{c_hirzRNS}.

The complement of the torus consists of four smooth rational curves: the zero and pole locus of the rational function $x$, which we denote $F$ and $F'$ respectively; the closure of the $x$-axis in $U_1$, which we denote by $D$ and call the \emph{directrix}, and the closure of $V(y^{-1}) \subset U_2$, which we denote by $D'$. The rational functions $x,y$ show that $F \sim F'$ and $D' \sim D + \delta F$. In fact, the Picard group of $\fd$ is free on generators $D,F$ (see \cite[p. 70]{fulton}), but we do not need this fact.

There is a projection $\pi: \fd \rightarrow D$ to the directrix given by $(x,y) \mapsto (x,0)$ on $U_1$; similar expressions may be found in the other charts. The fibers of this projection are linearly equivalent smooth rational curves (including $F$ and $F'$), which we will call the \emph{ruling lines}.

The global sections of the line bundle $\cO_{\FF_\delta} (\alpha D + \beta F)$, regarded as rational functions regular away from the divisors $D$ and $F$, must be regular on the torus (i.e. lie in $\CC[ x^{\pm1}, y^{\pm 1} ]$), and analyzing orders of vanishing on the boundary shows that (cf. \cite[Lemma on p. 66]{fulton})
\begin{eqnarray*}
H^0(\FF_\delta, \cO_{\FF_\delta}(\alpha D + \beta F)) &=& \sum_{(i,j) \in P_{\alpha,\beta}} \mathrm{span} (x^{-i} y^{-j})\\
\textrm{where } P_{\alpha,\beta} &=& \left\{ (i,j) \in \ZZ^2:\ 0 \leq j \leq \alpha \mbox{ and } j \delta \leq i \leq \beta \right\}.
\end{eqnarray*}

\begin{cor} \label{c_h0}
The dimension of the space of global sections of $\cO_{\FF_\delta}(\alpha D + \beta F)$ is
$$
h^0(\fd, \cO_{\fd}(\alpha D + \beta F)) = 
\displaystyle \sum_{j=0}^\alpha \max (0, \beta + 1 - j \delta ),
$$
where this sum is understood to be $0$ when $\alpha < 0$.
\end{cor}

The following Corollary follows in a straightforward manner (cf. \cite[Theorem 2.17, Corollary 2.18]{hartshorne}).

\begin{cor} \label{c_bpf}
The complete linear series $|\alpha D  + \beta F|$ is nonempty if and only if $\alpha$ and $\beta$ are both nonnegative. The complete linear series $| \beta F|$ ($\beta \geq 0$) consists of all sums of $\beta$ ruling lines. The complete linear series $|\alpha D + \beta F|$, where $\alpha,\beta \geq 0$, is base point free if $\beta \geq \delta \alpha$, and has base locus equal to $D$ if $0 \leq \beta < \delta \alpha$.
\end{cor}

\begin{cor} \label{c_hirzRNS}
If $a,b$ are nonnegative integers with $b-a = \delta$, then the map $\fd \rightarrow \scroll_{a,b}$ described above is $\phi_{|D+ b\cdot F|}$.
\end{cor}

The curves $F$ and $D$ intersect transversely at one point, $F$ is disjoint from $F'$, and $D$ is disjoint from $D'$. The intersection products $D^2 = -\delta,\ D.F = 1,\ F^2 = 0$ follow from this. The canonical divisor class of $\fd$ is $-2D - (\delta+2) F$, since the form $d \log x \wedge d \log y$ has divisor $-D-D'-F-F'$ (cf. \cite[p. 85]{fulton}), and $\chi(\fd, \cO_{\fd}) = 1$ since $\fd$ is a rational surface, hence the Riemann-Roch formula for line bundles on surfaces \cite[I.12]{beau} gives the following formula, after some simplification.

\begin{equation} \label{eq_chi}
\chi \left(\fd, \cO_{\fd} ( \alpha D + \beta F) \right) = (\alpha+1)(\beta+1) - \frac12 \delta (\alpha+1)\alpha
\end{equation}

\begin{cor} \label{c_h1}
If $\alpha,\beta$ are integers with $\alpha \geq -1$, and $C \sim \alpha D + \beta F$, then
$$
h^1(\fd, \cO_{\fd} (C) )= \sum_{j=0}^\alpha \max(0, j \delta - \beta - 1).
$$
This is $0$ if and only if $C.D \geq -1$, and it is equal to $-C.D-1$ if $-\delta -1 \leq C.D \leq -1$.
\end{cor}
\begin{proof}
The right side of equation (\ref{eq_chi}) can be rewritten $\sum_{j=0}^{\alpha} (\beta + 1 - j \delta)$. Corollary \ref{c_h0}, with Serre duality and the fact that the canonical divisor of $\fd$ is $-2D - (\delta+2) F$, shows that $\cO_{\fd}(\alpha D + \beta F)$ has vanishing second cohomology. Now use Corollary \ref{c_h0}. The last sentence follows from the fact that $(\alpha D + \beta F) . D = \beta - \alpha \delta$.
\end{proof}

\subsection{Fully inflected curves on $\FF_\delta$} \label{ss_fdCurves}
Throughout this subsection, we fix integers $r,m,\eps$, and $\delta$, such that $r \geq 3$, $r-1 \geq \delta$, $r-1 \equiv \delta \pmod{2}$, $m \geq 2$, and $0 \leq \eps \leq r-1$. We also assume that if $m=2$, then $\eps \geq 2$. 

We denote by $d,g$ the integers $m(r-1) + \eps + 1$ and $\binom{m}{2} (r-1) + m \eps$, respectively, and $H_r,\Gamma$ denote the following two divisors.
\begin{eqnarray*}
H_r &=& D + \frac12(r-1+\delta) F\\
\Gamma &=& (m+1) H_r - (r-2-\eps) F
\end{eqnarray*}

The divisor $H_r$ is the pullback of a hyperplane under $\phi_{|H_r|}$ (Corollary \ref{c_hirzRNS}); observe that $H_r . H_r = r-1$.

We will analyze the following variety. It is the analog of $\Phi^\delta_{r,m,\eps}$ for a fixed Hirzebruch surface rather than a variable scroll.

\begin{defn} \label{d_Psirmed}
Denote by $\psirmed$ the sub-variety of $|\Gamma| \times |H_r| \times \fd$ consisting of triples $(C,\zeta,p)$, where $C \sim \Gamma$ is a smooth curve, $\zeta$ is an element of $|H_r|$, and $p$ is a point of $C$, such that the pullback of the divisor $\zeta$ to $C$ is equal to the divisor $d \cdot p$, and such that the linear series $|H_r|$ includes elements meeting $C$ to order $0,1,\cdots,r-1$ at $p$.
\end{defn}

\begin{lemma} \label{l_zetaSmooth}
For all points $(C,\zeta,p) \in \psirmed$, the divisor $\zeta$ is a smooth.
\end{lemma}
\begin{proof}
We claim that $\zeta$ is irreducible. This will show that it is smooth, since $\zeta . F = 1$, hence $\zeta$ meets all ruling lines transversely and must have a $1$-dimensional tangent space at all points.

Suppose that $\zeta$ is reducible. Then one irreducible component has class $D + \ell F$ for some $\ell \in \ZZ$, while all other irreducible components are supported on ruling lines, by Corollary \ref{c_bpf}. So $\zeta$ includes \emph{the} ruling line through $p$ with some positive multiplicity. Since $C$ can meet this ruling line at only one point, it meets it to order $C.F = m+1$ at $p$. By moving one copy of the ruling line to a different location, we obtain another divisor $\zeta' \sim \zeta$ meeting $C$ to order $d - m - 1$ at $p$. This vanishing order must lie in $\{0,1,\cdots, r-1,d\}$, hence $d-m-1 \leq r-1$. Since $d -1 = m(r-1) + \eps$, this implies that $(m-1)(r-2) + \eps \leq 1$. But this is impossible, since $r \geq 3$ and $m+ \eps \geq 3$.
\end{proof}

\begin{lemma} \label{l_contacts}
If $C$ is a smooth irreducible curve in $|\Gamma|$, $\zeta$ is a smooth element of $|H_r|$, and the pullback of $\zeta$ to $C$ is $d \cdot p$, then there exist other elements of $|H_r|$ meeting $C$ to orders $0,1,\cdots,r-1$ at $p$. In other words, $(C,\zeta,p) \in \psirmed$.
\end{lemma}
\begin{proof}
By Corollary \ref{c_h1}, $h^1(\fd,\cO_{\fd}) = 0$, hence the map $H^0(\fd, \cO_{\fd}(H_r)) \rightarrow H^0(\zeta, \cO_\zeta(r-1) )$ is surjective. In other words, the linear series $|H_r|$ on $\fd$ pulls back to the complete linear series of degree $r-1$ on $\zeta \cong \PP^1$. So there exist elements of $|H_r|$ meeting $\zeta$ to orders $0,1,2,\cdots,r-1$ at $p$. Since $\zeta$ meets the smooth curve $C$ to order $r$ at $p$, these same elements of $|H_r|$ also meet $C$ to orders $0,1,\cdots,r-1$ at $p$.
\end{proof}

\begin{cor} \label{c_psiDesc}
A triple $(C,\zeta,p) \in |\Gamma| \times |H_r| \times \fd$ lies in $\psirmed$ if and only if $\zeta$ and $C$ are smooth, $C$ is irreducible, and the pullback of the divisor $C$ to $\zeta$ is equal to the divisor $d \cdot p$ on $\zeta$.
\end{cor}

\begin{rem} \label {r_carvalho}
Whereas our analysis has led us to curves $C$ on scrolls meeting a hyperplane section $\zeta$ at a single point $p$, Carvalho \cite{carvalho} has studied an analogous situation: Weierstrass semigroups of points $p$ on curves $C$ on scrolls where a \emph{ruling line} meets $C$ at the point $p$ alone.
\end{rem}

\begin{lemma} \label{l_h1bound}
Let $\zeta$ be a smooth element of $|H_r|$, and $p \in \zeta$ a point. The linear map
$$
H^0(\fd, \cO_{\fd}(\Gamma)) \rightarrow H^0(\zeta, \cO_\zeta(d \cdot p))
$$
induced by the short exact sequence 
$$
0 \rightarrow \cO_{\fd}(\Gamma-\zeta) \rightarrow \cO_{\fd}(\Gamma) \rightarrow \cO_{\zeta} (d \cdot p) \rightarrow 0
$$
has cokernel of dimension less than or equal to $\max(\delta-2,0)$.
\end{lemma}
\begin{proof}
The dimension of this cokernel is equal to the dimension of the image of the boundary map $H^0(\zeta, \cO_\zeta(d \cdot p ) ) \rightarrow H^1(\fd, \cO_{\fd}(\Gamma - \zeta))$. Hence since $\zeta \sim H_r$ it suffices to show that $h^1(\fd, \cO_{\fd}(\Gamma-H_r)) \leq \max(\delta-2,0)$. By Corollary \ref{c_h1}, it suffices to show that $(\Gamma-H_r) . D \geq 1-\delta$. Computing this intersection number and rearranging yields
\begin{eqnarray*}
(\Gamma - H_r) . D 
&=& \frac12 (m-2) (r-1-\delta) + 1 + \eps - \delta.
\end{eqnarray*}
Since $m \geq 2$, $\delta \leq r-1$, and $\eps \geq 0$, it follows that $(\Gamma-H_r).D \geq 1-\delta$.
\end{proof}

\begin{prop} \label{p_dimpsi1}
If $\delta \in \{0,1\}$, then $\psirmed$ is irreducible of dimension $g+2m+\eps+4$.
\end{prop}
\begin{proof}
Fix a smooth $\zeta \in |H_r|$ and $p \in \zeta$. Lemma \ref{l_h1bound} shows that $|\Gamma|$ cuts out a complete linear series on $\zeta$, and in particular there is a codimension $d$ sub-series of $|\Gamma|$ consisting of divisors $C$ that either contain $\zeta$ entirely or meet it to order $d$ at $p$. Denote this linear series by $\mathfrak{d}$.

We claim that the only base point of $\mathfrak{d}$ is $p$ itself. First, observe that $\mathfrak{d}$ includes the linear series $\zeta + |\Gamma - \zeta|$ of those elements containing $\zeta$ entirely.  As in the proof of Lemma \ref{l_h1bound}, we have $(\Gamma - \zeta). D \geq 1-\delta \geq 0$. Hence by Corollary \ref{c_bpf}, $|\Gamma - \zeta|$ is base point free. This shows that any base point of $\mathfrak{d}$ must lie on $\zeta$. Next, since $H^0(\cO_{\fd}(\Gamma)) \rightarrow H^0(\cO_\zeta(d \cdot p))$ is surjective, $\mathfrak{d}$ necessarily contains some members that intersect $\zeta$ only at the point $p$. Therefore $p$ is the only base point of $\mathfrak{d}$.

Bertini's theorem \cite[Corollary 10.9]{hartshorne} implies that a general member of $\mathfrak{d}$ is nonsingular except possibly at $p$. On the other hand, there exist elements of $\mathfrak{d}$ consisting of $\zeta$ itself plus a divisor that does not meet $p$, and these members are smooth at $p$. So for a general member of $\mathfrak{d}$, $p$ is a smooth point. Therefore a general element of $\mathfrak{d}$ is a smooth curve meeting $\zeta$ to order $d$ at $p$. All such curves are necessarily irreducible, since otherwise all components would meet at $p$ and the curve would not be smooth.

Equation (\ref{eq_chi}) and a rearrangement of terms shows that
$$
\chi(\fd, \cO_{\fd}(\Gamma) )= g + d -r + 2m + \eps + 4.
$$
As remarked above, $(\Gamma - \zeta) . D \geq 0$, so certainly $\Gamma . D \geq 0$, hence $\Gamma$ is non-special and
$\dim |\Gamma| = g+d -r + 2m + \eps + 3.$
Therefore $\dim \mathfrak{d} = g-r+2m+\eps+3$.

By Corollary \ref{c_psiDesc}, a dense open subset of $\mathfrak{d}$ is a fiber of the forgetful map
$$
\psirmed \rightarrow \{ (\zeta,p):\ \zeta \in |H_r| \mbox{ smooth, } p \in \zeta \}.
$$
Therefore this map is surjective, and all fibers are irreducible of dimension $g-r+2m+\eps+3$. The target of this map is irreducible of dimension $r+1$. Hence $\psirmed$ is irreducible of dimension $g+2m+\eps+4$.
\end{proof}

\begin{prop} \label{p_dimpsi2}
If $\delta \geq 2$, then either $\psirmed$ is empty, or $$\dim \psirmed \leq g+2m+\eps+4 + (\delta-2).$$
\end{prop}
\begin{proof}
Fix a smooth $\zeta \in |H_r|$ and $p \in \zeta$. As in the proof of Proposition \ref{p_dimpsi1}, let $\mathfrak{d}$ denote the sub series of $|\Gamma|$ consisting of divisors either containing $\zeta$ entirely or meeting it to order $d$ at $p$. Lemma \ref{l_h1bound} implies that the codimension of $\mathfrak{d}$ in $|\Gamma|$ is at least $d+2 - \delta$.

By Corollary \ref{c_psiDesc}, the fiber over $(\zeta,p)$ of the forgetful map (as in the proof of Proposition \ref{p_dimpsi1}) is either empty or an open subset of $\mathfrak{d}$. If it is nonempty, then $\Gamma . D \geq 0$, since $|\Gamma|$ has elements that are smooth irreducible curves not equal to $D$ itself. So we deduce in this case, as in the proof of Proposition \ref{p_dimpsi1}, that $\dim |\Gamma| = g+d-r+2m+\eps+3$. It follows that, if $\Psi^\delta_{r,m,\eps}$ is nonempty, then its dimension is at most $\dim |\Gamma| - (d+2 - \delta) + r+1 = g + 2m + \eps + 2 + \delta$, as desired.
\end{proof}

\section{The irreducible components of $\cM^S_{g,1}$} \label{s_msg}

We now use the analysis of the previous two sections to determine the irreducible components of $\cM^{S_{r,d}}_{g,1}$. This amount to making dimension comparisons between parameter spaces on a fixed surface, in the Hilbert scheme, and in $\cM_{g,1}$.

\begin{prop} \label{verParams}
The parameter space of Veronese surfaces in $\PP^5$ is irreducible of dimension $27$.
\end{prop}
\begin{proof}
All Veronese surfaces in $\PP^5$ are projectively equivalent, every automorphism of a Veronese surface is induced by an automorphism of $\PP^5$, and there are no nontrivial automorphisms of $\PP^5$ fixing all points of a Veronese surface. Therefore the space of Veronese surfaces has dimension $\dim \Aut \PP^5 - \dim \Aut \PP^2 = 35-8 = 27$.
\end{proof}

Combining this with Proposition \ref{p_verPsiDim}, we obtain the following.

\begin{cor} \label{c_verDimPhi}
The parameter space $\Phi^{\textrm{ver}}_{5,d}$ is irreducible of dimension $g + 2m + \frac{\eps-1}{2} + 33$.
\end{cor}

\begin{prop} \label{p_scrollParams}
The parameter space of scrolls of invariant $\delta$ in $\PP^r$ is irreducible of dimension $(r+1)^2 -6 - \delta$ if $\delta > 0$, and $(r+1)^2 - 7$ if $\delta = 0$.
\end{prop}
\begin{proof}
Suppose first that $\delta < r-1$. Let $a = \frac12(r-1-\delta)$ and $b = \frac12(r-1+\delta)$. A particular scroll of invariant $\delta$ is determined by a rational normal curve of degree $a$ (spanning an $a$-plane in $\PP^r$), a rational normal curve of degree $b$ (spanning a $b$-plane), and an isomorphism between them. The space of choices of these data has dimension $(r+1)(a+1)-4 + (r+1)(b+1)-4 + 3 = (r+1)^2 - 5$. These two rational curves have classes $D$ and $D'$ in the notation of Section \ref{s_hirzebruch}. A different choice of the two rational curves gives the same surface in $\PP^r$ if and only if they are, respectively, images under the map $\textbf{F}_{\delta} \rightarrow S_{a,b}$ of curves in the classes $D$ and $D'$, respectively. Hence the dimension of this parameter space is $(r+1)^2 - 5 - \dim |D| - \dim |D'|$. Corollary \ref{c_h0} shows that if $\delta > 0$ then $\dim |D| = 0$ and $\dim |D'| = \delta +1$, while if $\delta = 0$ then $\dim |D| = \dim |D'| = 1$.

If $\delta = r-1$, the argument is essentially the same, except that rather than choosing a degree $a$ rational curve, one chooses a single point in $\PP^r$, and there is no need to choose an isomorphism between the two rational curves. These differences cancel each other to give the same expression for the number of parameters.
\end{proof}

\begin{cor} \label{c_dimPhi}
The parameter space $\Phi_{r,m,\eps}$ is irreducible of dimension $g+2m+\eps -2 + \dim \Aut \PP^r$, and for $\delta \in \{0,1\}$ ($\delta \equiv r-1 \mod{2}$), $\Phi^{\delta}_{r,m,\eps}$ is dense in $\Phi_{r,m,\eps}$.
\end{cor}

\begin{proof}
Propositions \ref{p_dimpsi1} and \ref{p_scrollParams} imply that for $\delta \in \{0,1\}$ (whichever is congruent to $r-1$ modulo $2$), $\Phi^\delta_{r,m,\eps}$ is irreducible of the desired dimension. On the other hand, Propositions \ref{p_dimpsi2} and \ref{p_scrollParams}, and Lemma \ref{l_dimbound},  show that for $\delta > 1$, any irreducible component of $\Phi^\delta_{r,m,\eps}$ has codimension at least $(\delta-1)-(\delta-2) = 1$ in $\Phi_{r,m,\eps}$. Hence no irreducible component of $\Phi_{r,m,\eps}$ lies in the closure of $\Phi^\delta_{r,m,\eps}$ for $\delta > 1$.

Therefore there is only one irreducible component, the closure of the locus corresponding to curves on balanced scrolls, of the claimed dimension.
\end{proof}

\begin{defn}
Denote by $\cw_{r,m,\eps}$ the image of $\Phi_{r,m,\eps}$ in $\cM_{g,1}$, and by $\cw^{\textrm{ver}}_{r,d}$ the image of $\Phi^{\textrm{ver}}_{r,d}$ in $\cM_{g,1}$. 
\end{defn}

We can now state and prove the main theorem of the paper, characterizing the irreducible components of $\cM^{S_{r,d}}_{g,1}$.

\begin{thm} \label{t_main}
Let $r,d$ be integers with $r \geq 3$ and $d \geq 2r+1$. Let $m,\eps$ be the quotient and remainder when $d-1$ is divided by $r-1$. Then one irreducible component of $\cM^{S_{r,d}}_{g,1}$ is $\cw_{r,m,\eps}$, which has dimension $g+2m+\eps-2$. In the following two cases, there is a second irreducible component, while in all other cases $\cM^{S_{r,d}}_{g,1}$ is irreducible.
\begin{enumerate}
\item If $r \geq 4$ and $r-1$ divides $d-1$, then $\cw_{r,m-1,r-1}$ is a second irreducible component, and its dimension is $r-3$ greater than the other's.
\item If $r = 5$ and $d$ is even, then the closure of $\cw^{\textrm{ver}}_{r,d}$ in $\cM^{S_{r,d}}_{g,1}$ is a second irreducible component, and its dimension is $\frac{\eps+1}{2}$ less than the other's.
\end{enumerate}
\end{thm}

\begin{proof}
Corollaries \ref{c_verDimPhi} and \ref{c_dimPhi} show that the loci $\Phi_{r,m,\eps}$ and $\Phi_{r,d}$ (for all choices of $m,\eps,r,d$) are all irreducible, and give their dimensions (each of which exceeds the claimed dimension of the component in $\cM^{S_{r,d}}_{g,1}$ by $\dim \Aut \PP^r$). None of these loci can have closure contained in the closure of another locus, since then Theorem \ref{t_ciliberto} would imply that all points in one of these loci correspond to curves on singular scrolls, but these are a proper subset of the closure of $\Phi^{\textrm{ver}}_{r,d}$, and also of $\Phi_{r,m,\eps}$ by Corollary \ref{c_dimPhi}. So the various components $\Phi_{r,m,\eps}$ and the closures of the components $\Phi^{\textrm{ver}}_{5,d}$ constitute the irreducible components of the varieties $\Phi_{r,d}$.

Corollary \ref{c_mapToMg} now shows that the irreducible components of $\cM^{S_{r,d}}_{g,1}$, and their dimensions, are as claimed.
\end{proof}

From this we deduce the Theorem from the introduction, which is to say the inequalities on the dimensions of the components of $\cM^{S_{r,d}}_{g,1}$.

\begin{proof}[Proof of Theorem \ref{t_intro}]
Part (2) is included in Theorem \ref{t_main}. To prove the inequalities in part (1) for both types of irreducible components coming from curves on scrolls, it suffices to check, for any integers $r,m,\eps$ with $r \geq 3$, $m \geq 2$, $0 \leq \eps \leq r-1$, and $\eps \geq 2$ in the case $m =2$, the inequalities
$$
5 \leq 2m + \eps \leq \frac13 \left[ \binom{m}{2} (r-1) + m \eps \right] + 4.
$$
Note that the two different possible dimensions of components are accounted for since $\eps$ can be either $0$ or $r-1$. The stronger lower bound $6 \leq 2m + \eps$ follows immediately from our assumptions. The upper bound is equivalent to $0 \leq \binom{m}{2} (r-1) - 6m + (m-3)\eps + 12$. In the case $m=2$, the right side is equal to $(r-1) - \eps$, which is nonnegative. In the case $m -3\geq 0$, the right side is at least $\binom{m}{2} \cdot 2 - 6m + 12 = (m-3)(m-4) \geq 0$.

For the component from curves on Veronese surfaces (when $r=5$), it is only necessary to check the lower bound, which amounts to $5 \leq 2m + \frac{\eps-1}{2}$. For $m=2$ the only possibility is $\eps = 3$, for which equality is obtained. For $m \geq 3$ the inequality follows since $\eps-1 \geq 0$.
\end{proof}

\begin{rem} \label{r_thmSharp}
Checking equality cases in the proof above shows that both bounds are sharp. The lower bound $g+3 \leq \dim X$ occurs only in one case: $\cw^{\textrm{ver}}_{5,12}$, for the Veronese component. However, $\dim X = g+4$ is attained in infinitely many cases: $\cw^{\textrm{ver}}_{5,14}$ and when $d=2r+1$ or $d=3r-2$ (on $\cw_{r,m,\eps}$). The upper bound $\dim X \leq \frac43 g + 2$ is also sharp; it occurs when $d = 3r-2$ for any $r$ (for the larger of the two components), and when $(r,d) = (3,8)$ or $(3,9)$.
\end{rem}

\section{Effective weights of Castelnuovo semigroups} \label{s_ew}

We explain in this section our original interest in Castelnuovo semigroups.

In \cite{pfl}, we prove that for every numerical semigroup $S$ of genus $g$, every component of $\cM^S_{g,1}$ has dimension at least $3g-2 - \ewt(S)$, where $\ewt(S)$ was defined in the introduction, and prove the existence of \emph{effectively proper components}, that is components where equality is achieved, for all semigroups such that $\ewt(S) \leq g-1$. Several specific families of numerical semigroups are provided in that paper, where $\cM^S_{g,1}$ has codimension larger than $g$ and is effectively proper.

We were originally led to study Castelnuovo semigroups by searching for semigroups for which the effective weight bound was likely to not hold with equality. A natural way to search for such examples is to find those semigroups of maximum effective weight in each genus. This led to the following conjecture (Conjecture 1.7 in \cite{pfl}), which has been have verified for all semigroups of genus up to $50$ by a computer search\footnote{Source code in C++ is available on the author's academic website; the search took approximately 17 hours on a 3.4GHz Intel i7-3770 CPU.}.

\begin{conj} \label{conj_ew}
For any numerical semigroup $S$ of genus $g$, $\ewt(S) \leq \left\lfloor \frac{(g+1)^2}{8} \right\rfloor$. For $g \geq 10$, the only equality cases are Castelnuovo semigroups.
\end{conj}

 \begin{figure}
 \begin{tabular}{cc}
 $
\begin{array}{|l|l|l|l|}
\hline
g & \ewt & \mbox{ gaps } & \mbox{ gens. } \\\hline
\hline
\multirow{1}{*}{$1$} & \multirow{1}{*}{0}& {1} & \langle 2,3 \rangle \\
\hline
\multirow{1}{*}{$2$} & \multirow{1}{*}{1}& {1,3} & \langle 2,5 \rangle \\
\hline
\multirow{2}{*}{$3$} & \multirow{2}{*}{2}& {1,2,5} & \langle 3,4 \rangle \\
&& {1,3,5} & \langle 2,7 \rangle \\
\hline
\multirow{3}{*}{$4$} & \multirow{3}{*}{3}& {1..3,7} & \langle 4..6 \rangle \\
&& {1,2,4,7} & \langle 3,5 \rangle \\
&& {1,3,5,7} & \langle 2,9 \rangle \\
\hline
\multirow{5}{*}{$5$} & \multirow{5}{*}{4}& {1..4,9} & \langle 5..8 \rangle \\
&& {1..3,5,9} & \langle 4,6,7 \rangle \\
&& {1..3,6,7} & \langle 4,5,11 \rangle \\
&& {1,2,4,5,8} & \langle 3,7,11 \rangle \\
&& {1,3,5,7,9} & \langle 2,11 \rangle \\
\hline
\multirow{4}{*}{$6$} & \multirow{4}{*}{6}& {1..4,8,9} & \langle 5..7 \rangle \\
&& {1..3,5,7,11} & \langle 4,6,9 \rangle \\
&& {1..3,6,7,11} & \langle 4,5 \rangle \\
&& {1,2,4,5,8,11} & \langle 3,7 \rangle \\
\hline
\multirow{2}{*}{$7$} & \multirow{2}{*}{8}& {1..5,10,11} & \langle 6..9 \rangle \\
&& {1..3,5,7,9,13} & \langle 4,6,11 \rangle \\
\hline
\end{array}
$
&
$
\begin{array}{|l|l|l|l|}
\hline
g & \ewt & \mbox{ gaps } & \mbox{ gens. } \\\hline
\hline
\multirow{2}{*}{$8$} & \multirow{2}{*}{10}& {1..6,12,13} & \langle 7..11 \rangle \\
&& {1..3,5,7,9,11,15} & \langle 4,6,13 \rangle \\
\hline
\multirow{4}{*}{$9$} & \multirow{4}{*}{12}& {1..7,14,15} & \langle 8..13 \rangle \\
&& {1..6,11..13} & \langle 7..10 \rangle \\
&& {1..5,9..11,17} & \langle 6..8 \rangle \\
&& {1..3,5,7,9,11,13,17} & \langle 4,6,15 \rangle \\
\hline
\multirow{1}{*}{$10$} & \multirow{1}{*}{15}& {1..7,13..15} & \langle 8..12 \rangle \\
\hline
\multirow{1}{*}{$11$} & \multirow{1}{*}{18}& {1..8,15..17} & \langle 9..14 \rangle \\
\hline
\multirow{1}{*}{$12$} & \multirow{1}{*}{21}& {1..9,17..19} & \langle 10..16 \rangle \\
\hline
\multirow{2}{*}{$13$} & \multirow{2}{*}{24}& {1..10,19..21} & \langle 11..18 \rangle \\
&& {1..9,16..19} & \langle 10..15 \rangle \\
\hline
\multirow{1}{*}{$14$} & \multirow{1}{*}{28}& {1..10,18..21} & \langle 11..17 \rangle \\
\hline
\multirow{1}{*}{$15$} & \multirow{1}{*}{32}& {1..11,20..23} & \langle 12..19 \rangle \\
\hline
\multirow{1}{*}{$16$} & \multirow{1}{*}{36}& {1..12,22..25} & \langle 13..21 \rangle \\
\hline
\multirow{2}{*}{$17$} & \multirow{2}{*}{40}& {1..13,24..27} & \langle 14..23 \rangle \\
&& {1..12,21..25} & \langle 13..20 \rangle \\
\hline
\multirow{1}{*}{$18$} & \multirow{1}{*}{45}& {1..13,23..27} & \langle 14..22 \rangle \\
\hline
\multirow{1}{*}{$19$} & \multirow{1}{*}{50}& {1..14,25..29} & \langle 15..24 \rangle \\
\hline
\end{array}
$
\end{tabular}
\caption{Maximum effective weights of semigroups of genus $g \leq 19$, and all semigroups achieving the maximum in each case.} \label{fig_maxewt}
\end{figure}

Figure \ref{fig_maxewt} illustrates this conjecture for genera up to $19$. It is easy to guess from these data which specific Castelnuovo semigroup achieves the maximum for genus $g \geq 10$. One can set
\begin{eqnarray*}
d &=& \frac14 ( 5g + 1 + 3 \eta )\\
r &=& \frac12 ( g + 1 + \eta ),
\end{eqnarray*}

where $\eta$ is an integer between $-2$ and $2$ inclusive such that $\eta \equiv g+1 \pmod{4}$ (there are two choices when $g \equiv 1 \pmod{4}$, and only one choice otherwise). Then $S_{r,d}$ is a genus $g$ Castelnuovo semigroup of effective weight $\frac18 (g+1)^2 - \frac18 \eta^2 = \left\lfloor \frac{(g+1)^2}{8} \right\rfloor$. When $g \geq 10$, these indeed are the only equality cases coming from Castelnuovo semigroups.

\begin{prop} \label{p_castMaxEw}
For any Castelnuovo semigroup $S_{r,d}$ of genus $g$, Conjecture \ref{conj_ew} holds. For $g \geq 10$, the only Castelnuovo semigroups that achieve equality $\ewt(S_{r,d}) = \left\lfloor \frac{(g+1)^2}{8} \right\rfloor$ are those described above.
\end{prop}
\begin{proof}
The effective weight of $S_{r,d}$ is equal to $r(g-d+r)$, since all but $d-r$ of the gaps lie below all generators, and all other gaps lie above all generators. Let $m,\eps$ be the quotient and remainder when $d-1$ is divided by $r-1$. 

Consider first the case $m=2$. Then $g = r + 2 \eps -1$ and $d = 2r + \eps -1 $. The quantity $\frac18 (g+1)^2 - \ewt(S_{r,d})$ is equal to $\frac18 (r - 2\eps)^2$, which is obviously nonnegative. Furthermore, the floor of this quantity is $0$ if and only if $2\eps -2 \leq r \leq 2 \eps + 2$. Let $\eta = r - 2 \eps$; we have shown that the equality case of Conjecture \ref{conj_ew} is equivalent to $|\eta| \leq 2$. Some algebra shows that $g = 4 \eps + \eta -1$, hence $\eta \equiv g+1 \pmod{4}$, $d = \frac14 (5g+1+3\eta)$ and $r = \frac12 (g+1+\eta)$. So the proposition holds in the case $m=2$.

Now suppose $m \geq 3$. We will show that if $g \geq 10$, then $S_{r,d}$ has effective weight striclty less than $\left\lfloor \frac{(g+1)^2}{8} \right\rfloor$. It is straightforward to enumerate the Castelnuovo semigroups with $m \geq 3$ and genus $g \leq 9$ to check the Proposition in those cases (in three of them, namely $(r,d) = (3,7),(3,8)$, and $(4,10)$, equality occurs in Conjecture \ref{conj_ew}; the cases $(3,8)$ and $(4,10)$ give genus $9$ semigroups, so we must assume $g \geq 10$ in order to prove the strict inequality). 

Algebraic manipulation shows that $\ewt(S_{r,d}) = \frac{2(m-1)}{m^2} \left( \frac{mr}{2} \right) \left( g + \frac{m-mr}{2} \right)$. By the arithmetic-geometric mean inequality, this is at most $\frac{m-1}{2m^2} \left( g + \frac{m}{2} \right)^2$. We claim that this, in turn, is at most $\frac{1}{9} (g + \frac32)^2$ (with equality when $m=3$). To see this, consider the function $f(x) = \frac{x-1}{2x^2}(g+ \frac12 x)^2$. The logarithmic derivative $f'(x) / f(x)$ of $f(x)$ is equal to $\frac{(x-g)^2 - (g^2-4g)}{(x-1)x (x+2g)}$, which is negative for $3 \leq x \leq g$ (this requires assuming that $g \geq 5$). Since $m \leq g$, it follows that $f(m) \leq f(3) = \frac{1}{9}(g+\frac32)^2$. For all $g \geq 10$,  $\frac19 (g+\frac32)^2 < \left\lfloor \frac{(g+1)^2}{8} \right\rfloor$, hence $\ewt(S_{r,d}) < \left\lfloor \frac{(g+1)^2}{8} \right\rfloor$.
\end{proof}

Indeed, the Castelnuovo semigroups of maximum effective weight do provide examples of numerical semigroups for which $\dim \cM^S_{g,1}$ exceeds the prediction of the effective weight. In fact, most other Castelnuovo semigroups (not just those of maximum effective weight) also provide such examples.

\begin{prop} \label{p_effproper}
For a Castelnuovo semigroup $S = S_{r,d}$, $\cM^S_{g,1}$ has an effectively proper component if either $d \leq 2r+1$, or $d=2r+2$ and $r \in \{4,5\}$. Otherwise, all components of $\cM^S_{g,1}$ have dimension strictly greater than $3g-2-\ewt(S)$.
\end{prop}
\begin{proof}
As mentioned in the proof of Proposition \ref{p_castMaxEw}, $\ewt(S) = r(g-d+r)$. One can check, using the cases listed in Section \ref{s_intro}, that if $d \leq 2r$ or $r=2$ then $\cM^{S_{r,d}}_{g,1}$ is irreducible and effectively proper. So it suffices to consider $d \geq 2r+1$. 

Denote by $m,\eps$ the quotient and remainder when $d-1$ is divided by $r-1$. Then one can compute using the expression in Theorem \ref{t_main} that
$\dim \cw_{r,m,\eps} - (3g-2-\ewt(S_{r,d}))$
is equal to
\begin{equation*}
\binom{m-2}{2} (r-1)(r-2) + (r-3) \left( (r-2)(m-2) - 2\right) + \eps \left( (r-2)(m-1) -1 \right).
\end{equation*}

Consider first the case $m=2$. Note that this implies $r \geq 4$. Then this expression simplifies to $(\eps-2)(r-3)$. This is nonnegative, and equal to $0$ if and only if $\eps = 0$, i.e. $d = 2r+1$.

Now consider the case $m \geq 3$.  All three terms in the sum are necessarily nonnegative in this case. The first is equal to $0$ if and only if $m=3$. Given that $m=3$, the second is equal to $0$ if and only if $r=3$ or $r=4$. The third is $0$ if and only if $\eps = 2$. So the whole expression is $0$ if and only if $d = 3r-2$ and $r=3$ or $r=4$.

Therefore $\dim \cw_{r,m,\eps} \geq 3g-2 - \ewt(S_{r,d})$ with equality if and only if either $d = 2r+1$ (this includes the case $r=3,\ d=3r-2$) or $(r,d) = (4,10)$. In the case where $r-1$ divides $d-1$, the second component of $\cM^{S_{r,d}}_{g,1}$ has dimension strictly larger than the first, so we do not need to consider it.

Now consider the components $\cw^{\textrm{ver}}_{5,d}$, in the case where $r=5$ and $d$ is even. By Theorem \ref{t_main}, this component is effectively proper if and only if the expression above is equal to $\frac{\eps+1}{2}$, i.e.
$$
12 \binom{m-2}{2} + 2(3m - 8 ) + \eps(3m - 4) = \frac{\eps+1}{2}.
$$
If $m \geq 3$, then the left side is at least $2 + 5 \eps$, which is strictly greater than the right side. If $m=2$, then $\eps$ is necessarily $3$, and indeed the equation holds. So $\cw^{\textrm{ver}}_{5,d}$ is effectively proper if and only if $d=2r+2 = 12$.
\end{proof}

With this in mind, we are naturally led to the question: can the effective weight be improved to accommodate Castelnuovo semigroups?

\begin{qu} \label{q_betterWeight}
Is there  a combinatorial quantity associated to every numerical semigroup that gives an upper bound on the codimension of all components of $\cM^S_{g,1}$ for all numerical semigroups $S$ of genus $g$, is equal to $\ewt(S)$ when $\ewt(S) \leq g-1$, and is equal to the maximum codimension of a component of $\cM^S_{g,1}$ for all Castelnuovo semigroups $S$?
\end{qu}


\bibliography{main}{}

\begin{thebibliography}{ACGH85}

\bibitem[ACGH85]{ACGH}
E.~Arbarello, M.~Cornalba, P.~Griffiths, and J.~Harris.
\newblock {\em Geometry of algebraic curves, {V}olume {I}}.
\newblock Springer-Verlag, 1985.

\bibitem[Bea96]{beau}
Arnaud Beauville.
\newblock {\em Complex algebraic surfaces}.
\newblock Cambridge University Press, 1996.

\bibitem[Bul13]{bullock}
Evan Bullock.
\newblock Subcanonical points on algebraic curves.
\newblock {\em Transactions of the American Mathematical Society},
  365(1):99--122, 2013.

\bibitem[Car02]{carvalho}
C{\i}cero Carvalho.
\newblock Weierstrass gaps and curves on a scroll.
\newblock {\em Contributions to Algebra and Geometry}, 43(1):209--216, 2002.

\bibitem[Cil87]{ciliberto}
Ciro Ciliberto.
\newblock On the {H}ilbert scheme of curves of maximal genus in a projective
  space.
\newblock {\em Mathematische Zeitschrift}, 194(3):351--363, 1987.

\bibitem[CT03]{ct}
Cicero Carvalho and Fernando Torres.
\newblock On numerical semigroups related to covering of curves.
\newblock {\em Semigroup Forum}, 67(3), 2003.

\bibitem[EH87]{scrollitude}
David Eisenbud and Joe Harris.
\newblock On varieties of minimal degree.
\newblock {\em Proc. Sympos. Pure Math}, 46(1):3--13, 1987.

\bibitem[Ful93]{fulton}
William Fulton.
\newblock {\em Introduction to toric varieties}.
\newblock Princeton University Press, 1993.

\bibitem[GH78]{gh}
Phillip Griffiths and Joseph Harris.
\newblock {\em Principles of algebraic geometry}.
\newblock John Wiley \& Sons, 1978.

\bibitem[Har77]{hartshorne}
Robin Hartshorne.
\newblock {\em Algebraic Geometry}.
\newblock Springer, 1977.

\bibitem[HE82]{montreal}
Joe Harris and David Eisenbud.
\newblock Curves in projective space.
\newblock Les Presses de L'Universit\'{e} de Montr\'{e}al, 1982.

\bibitem[Pfl18]{pfl}
Nathan Pflueger.
\newblock On nonprimitive {W}eierstrass points.
\newblock {\em Algebra Number Theory}, 12(8):1923--1947, 2018.

\end{thebibliography}
\bibliographystyle{alpha}

\end{document}